\documentclass[11pt]{article}

\usepackage[utf8]{inputenc}
\usepackage[british]{babel}
\usepackage[a4paper, margin=2.5cm]{geometry}
\usepackage{amsmath, amsthm, amssymb, amsfonts}
\usepackage{mathtools}
\usepackage{thmtools}
\usepackage[shortlabels]{enumitem}
\usepackage[font={small, it}]{caption}
\usepackage{microtype}
\usepackage{xcolor}
\usepackage{mathabx}
\usepackage{tikz}
\usepackage{hyperref}
\usepackage{bm}

\newcommand{\cF}{\ensuremath{\mathcal F}}

\newcommand{\cP}{\ensuremath{\mathcal P}}

\renewcommand{\phi}{\varphi}
\renewcommand{\rho}{\varrho}

\let\setminus=\smallsetminus

\newcommand{\vv}{\mathbf{v}}

\declaretheorem[parent=section]{theorem}

\declaretheorem[sibling=theorem]{conjecture}

\setlist{itemsep=0.1em, topsep=0.1em, parsep=0.1em, partopsep=0.1em}

\hypersetup{
  colorlinks,
  linkcolor={red!60!black},
  citecolor={green!50!black},
  urlcolor={blue!80!black}
}

\colorlet{RoyalRed}{red!70!black}
\definecolor{RoyalBlue}{rgb}{0.25, 0.41, 0.88}
\definecolor{RoyalAzure}{rgb}{0.0, 0.22, 0.66}

\newlength{\bibitemsep}
\newlength{\bibparskip}
\let\oldthebibliography\thebibliography
\renewcommand\thebibliography[1]{%
  \oldthebibliography{#1}%
  \setlength{\parskip}{\bibitemsep}%
  \setlength{\itemsep}{\bibparskip}%
}

\title{Probabilistic intuition holds for a class of small subgraph games}
\author{
  Rajko Nenadov\thanks{Google Z\"urich. Email: \texttt{rajkon@gmail.com}.}
}
\date{}

\begin{document}
\maketitle

\begin{abstract}
Consider the following two-player game on the edges of $K_n$, the complete graph with $n$ vertices: Starting with an empty graph $G$ on the vertex set of $K_n$, in each round the first player chooses $b \in \mathbb{N}$ edges from $K_n$ which have not previously been chosen, and the second player immediately and irrevocably picks one of these edges and adds it to $G$. We show that for any graph $H$ with at least one edge, if $b < c n^{1/m(H)}$, where $c = c(H) > 0$ only depends on $H$ and $m(H)$ is the usual density function, then the first player can ensure the resulting graph $G$ contains $\Omega(n^{v(H)} / b^{e(H)})$ copies of $H$. The bound on $b$ is the best possible apart from the constant $c$ and shows that the density of the resulting graph for which it is possible to enforce the appearance of $H$ coincides with a threshold for the appearance in the Erd\H{o}s-R\'enyi random graph. This resolves a conjecture by Bednarska-Bzd\c ega, Hefetz, and \L uczak and provides a prominent class of games for which probabilistic intuition accurately predicts the outcome. The strategy of the first player is deterministic with polynomial running time, with the degree depending on the size of $H$.
\end{abstract}

\section{Introduction} \label{sec:introduction}

The Erd\H{o}s-R\'enyi random graph model $G(n, m)$, for some $n, m \in \mathbb{N}$ with $m \le \binom{n}{2}$, can be defined as follows: Start with an empty graph $G$ with $n$ vertices, and in each of the following $m$ rounds choose one edge uniformly at random among all edges in $\overline{G}$, the complement of $G$, and add it to $G$. One of the most prominent features of this model is that for any nontrivial monotone graph property $\cP$ there exists a function $m_0(n)$, called a \emph{threshold}, such that for $m = m(n)$ the following holds: 
$$
    \lim_{n \rightarrow \infty} Pr[ G(n, m) \in \cP] = \begin{cases}
    1, \; m \gg m_0(n)\\
    0, \; m \ll m_0(n).
    \end{cases}
$$
This remarkable result was proven by Bollob\'as and Thomason \cite{bllobas87threshold}. Here we are interested in the property of containing a (small) graph $H$, where `small' indicates that the size of $H$ does not depend on $n$. The problem of finding a threshold for the appearance of small subgraphs has already been studied in the seminal paper by Erd\H{o}s and R\'enyi \cite{erdos60renyi} and was fully resolved by Bollob\'as \cite{bollobas81smallsubgraphs}:

\begin{theorem}
For any graph $H$ with $e(H) \ge 1$, $m_0(H,n) = n^{2 - 1/m(H)}$ is a threshold for the property of containing $H$, where
$$
    m(H) = \max \left\{ \frac{e(H')}{v(H')} \colon H' \subseteq H, v(H') \ge 1 \right\}.
$$
\end{theorem}

Problems about the appearance of $H$ in $G(n,m)$ or, more generally, about the number of copies of $H$, still form a very active area of research. For a thorough treatment of the topic, we refer the reader to a monograph by Frieze and Karo\'nski \cite{introduction_random_graphs}.

Combining random graphs with the paradigm `power of multiple choices', Achlioptas suggested the following one-player game on the edges of $K_n$, the complete graph with $n$ vertices: Starting with an empty graph $G$ on the vertex set of $K_n$, in each step sample $r$ edges uniformly at random among all the edges of $K_n$ which have not been previously sampled; the player then picks one of these edges immediately and decisively and adds it to $G$. Given a graph $H$ and a constant $r \ge 1$, Krivelevich, Loh, and Sudakov \cite{krivelevich09avoiding} defined an \emph{avoidance-threshold} as a function $a_r(H, n)$ such that, with high probability, the player has a strategy which avoids the appearance of $H$ in the first $m \ll a_r(H, n)$ rounds, and no strategy can avoid the appearance after $m \gg a_r(H, n)$ rounds. Note that for $r = 1$, the player has no choices and the Achlioptas process matches $G(n, m)$, thus $m_0(H, n) = a_1(H, n)$. However, already for $r \ge 2$ we have $a_r(H, n) = n^{2 - 1/m(H) + \theta(H, r)}$ for a suitably defined $\theta(H, n) > 0$ \cite{krivelevich09avoiding,torsten11multiple}. In other words, the player can create an $H$-free graph with significantly larger density than $n^{-1/m(H)}$, whereas a random graph of such density with high probability contains (many) copies of $H$. Bednarska-Bzd\c ega, Hefetz, and \L uczak \cite{bednarsa14picker} conjectured that if random choices in the Achlioptas process are replaced by another player, then the appearance (and the number of copies) of $H$ in the resulting graph corresponds to the appearance in a random graph with the same density. Let us make this precise.

Consider the following two-player version of the Achlioptas process: Starting with an empty graph $G$ on the vertex set of $K_n$, in each round the first player chooses $b \in \mathbb{N}$ edges from $K_n$ which have not previously been chosen, and the second player chooses one of these edges and adds it to $G$. Games of this type were introduced by Beck \cite{beck02picker} under the name \emph{Picker-Chooser}, and in the recent literature they have been referred to as \emph{Waiter-Client} games (e.g.\ see \cite{bednarska13weight,bednarska16manipulative,bednarsa14picker,hefetz16wc,hefetz17wcham}). The first player is called \emph{Waiter} and the second \emph{Client}. We refer to the game with a parameter $b$ as the \emph{$b$-Waiter-Client} game.

As in the case of postponing the appearance of $H$ in the Achlioptas processes, the goal of Client is to avoid (or at least minimise the number of) copies of $H$ in $G$, whereas the goal of Waiter is to offer the edges such that this number is maximised.  Unlike in the Achlioptas process, where the player was able to `outsmart' the random sampling, Bednarska-Bzd\c ega, Hefetz, and \L uczak \cite{bednarsa14picker} suggested that Waiter can counter any strategy of Client: If the resulting graph has $m \gg n^{2 - 1/m(H)}$ edges ($b \ll n^{1/m(H)}$) then Waiter has a strategy which ensures the existence of $H$ in $G$, whereas for $m \ll n^{2 - 1/m(H)}$ ($b \gg n^{1/m(H)}$) Client can avoid $H$ throughout the whole game. In other words, the situation is essentially the same as if $G$ was a random graph $G(n, m)$ or, using the terminology from positional game theory, the $b$-Waiter-Client $H$-game follows the \emph{probabilistic intuition}. They proved the case $b \gg n^{1/m(H)}$ and conjectured that the other direction holds as well:

\begin{conjecture} \label{conj:probabilistic_intuition}
For every graph $H$ with $e(H) \ge 1$, there exist $c, \gamma > 0$ such that in a $b$-Waiter-Client game on $K_n$, for $n$ sufficiently large and $b \le c n^{1/m(H)}$, Waiter can ensure the resulting graph contains at least $\gamma n^{v(H)} / b^{e(H)}$ copies of $H$.
\end{conjecture}

As supporting evidence, Conjecture \ref{conj:probabilistic_intuition} was verified in \cite{bednarsa14picker} in the case where $H$ is a complete graph. A reader with some knowledge of positional game theory might find this conjecture  surprising at first, having in mind that in the \emph{Maker-Breaker} games, the most studied type of positional games, a prominent example where probabilistic intuition fails is precisely an $H$-game, for any fixed $H$ \cite{bednarska00randomstrategy}. Nonetheless, in this note we show Conjecture \ref{conj:probabilistic_intuition} is true.

\section{Proof of Conjecture \ref{conj:probabilistic_intuition}} \label{sec:proof}

Since Waiter's goal is to force many copies of $H$ in the resulting graph $G$, we consider a slight variation of the game where Waiter can decide, at any point, that the game is finished. Adding more edges certainly cannot decrease the number of copies of $H$, thus finishing the game earlier seems somewhat counter-intuitive. However, doing so will help to maintain the number of copies of certain graphs stemming from $H$ under control which, in turn, will imply that the copies of $H$ are well distributed.

To state the proof concisely we introduce some notation. For a given graph $H$, let $V(H) = \{h_1, \ldots, h_{v(H)}\}$. As Waiter can call for an early end of the game, instead of playing on $K_n$ we can restrict our attention to games on $K_q^H \subset K_n$, the complete \emph{$H$-partite} graph:  $V(K_q^H)$ is a disjoint union of sets $V_1, \ldots, V_{v(H)}$, each of size $q = \lfloor n / v(H) \rfloor$, and there is an edge between $v \in V_i$ and $u \in V_j$ iff $\{h_i, h_j\} \in H$.

Given an induced $J \subseteq H$, we denote with $H^J$ the graph obtained by overlapping two copies of $H$ on $J$. Formally, $H^J$ is the graph on the vertex set $V(J) \cup \bigcup_{h_i \notin V(J)} \{u_i, u_i'\}$ such that both $\phi \colon V(H) \rightarrow V(J) \cup \{u_i\}_{v_i \notin V(J)}$, given by $\phi(h_i) = h_i$ if $h_i \in V(J)$ and $\phi(h_i) = u_i$ otherwise, and $\phi' \colon V(H) \rightarrow V(J) \cup \{u_i'\}_{v_i \notin V(J)}$ given by $\phi(h_i) = h_i$ if $h_i \in V(J)$ and $\phi(h_i) = u_i'$ otherwise, are isomorphisms.

A \emph{canonical} copy of an induced $H' \subseteq H$ in $G \subseteq K_n^H$ is given by edge-preserving $\phi: V(H') \rightarrow V(G)$ such that $\phi(h_i) \in V_i$. Similarly, a canonical copy of $H^J$, for some induced $J \subseteq H$, is given by an injective (and edge-preserving) $\phi: V(H^J) \rightarrow V(K_n^H)$ such that $\phi(h_i) \in V_i$ if $h_i \in J$, and otherwise $\phi(u_i), \phi(u_i') \in V_i$. Given $G \subseteq K_n^H$, we let $N_G(H')$ and $N_G(H^J)$ denote the number of canonical copies of $H'$ and $H^J$ in $G$. For a pair of vertices $\vv \in K_n^H$, let $\deg_G(\vv, H')$ and $\deg_G(\vv, H^J)$ denote the number of canonical copies which contain both vertices in $\vv$.

Conjecture \ref{conj:probabilistic_intuition} follows from \ref{p1} in the following theorem applied with $\lfloor n/v(H) \rfloor$ (as $n$).

\begin{theorem} \label{thm:main_ind}
For every graph $H$ with $e(H) \ge 1$, there exist $c, T, \gamma > 0$ such that in a $b$-Waiter-Client game on $K_n^H$, for $b < cn^{1/m(H)}$ and $n$ sufficiently large, Waiter can ensure the resulting graph $G \subseteq K_n^H$ has the following properties for $p = 1/ b$:
\begin{enumerate}[(P1)]
    \item \label{p1} $N_G(H) \ge \gamma n^{v(H)} p^{e(H)}$,
    \item \label{p2} $N_G(H') \le T n^{v(H')} p^{e(H')}$ for every induced $H' \subseteq H$, and
    \item \label{p3} $N_G(H^J) \le T n^{2v(H) - v(J)} p^{2e(H) - e(J)}$ for every induced $J \subseteq H$.
\end{enumerate}
\end{theorem}
\begin{proof}
For a fixed integer $h \ge 2$, using induction on the number of edges we show that the claim holds for all graphs $H$ with $h$ vertices.

Consider some graph $H$ with exactly one edge $e = \{x,y\}$. We show that the statement holds for $c = 1/8$, $\gamma = 1/16$ and $T = 2$. The property \ref{p2} holds regardless which edges $V_x \times V_y$ Waiter offers, and $\ref{p1}$ holds as long as Waiter plays at least $\lceil \gamma n^2 / b \rceil < n^2 / (8b)$ rounds (that is, Waiter should not call for an early termination before round $n^2 / (8b)$). To satisfy \ref{p3}, it suffices to show that Waiter can play such that the resulting graph $G$ contains at most $Tn^3 p^2$ `cherries' in $V_x \times V_y$, where a \emph{cherry} consists of a vertex and two edges incident to it. Waiter proceeds in $n^2 / (8b)$ rounds, and in each round chooses $b$ new edges in $V_x \times V_y$ such that each offered edge is incident to at most $2np$ edges which have been picked by Client so far. Note that it could happen that, say, $p < 1/(2n)$, in which case Waiter only offers edges which are not incident to any of the edges currently in $G$ (consequently, the resulting graph is matching). In any event, the addition of each such edge increases the number of cherries by at most $2np$, thus, with room to spare, we get the desired bound on the total number of cherries. It remains to show that such $b$ edges always exist during the first $n^2 / (8b)$ rounds. Consider some round $i \le n^2 / (8b)$, and let us denote with $B_x \subseteq V_x$ and $B_y \subseteq V_y$ the set of vertices with more than $np$ incident edges in the current graph $G$, that is with degree larger than $np$. From $e(G) = i \le n^2 / (8b)$ we conclude $|B_x|, |B_y| \le n/8$. Any edge in $(V_x \setminus B_x) \times (V_y \setminus B_y)$ satisfies a desired property, and the choice of constants implies at least $b$ of them have not been previously offered. This concludes the proof of the base case. A similar, though slightly more general argument is used in the proof of the induction step.

Consider now some $H$ with $h$ vertices and $e(H) > 1$, and suppose the claim holds for $H_e = H \setminus e$, for an arbitrary edge $e = \{x, y\} \in H$. We treat $H_e$ as a spanning subgraph of $H$, meaning that even if $x$ or $y$ become isolated in $H_e$ we still keep them. In the first phase of the game, Waiter plays on $K_n^{H_e}$ such that the resulting graph $G_e \subseteq K_n^{H_e}$ satisfies \ref{p1}--\ref{p3} with respect to $H_e$, for some positive $c, T$, and $\gamma$. We now describe a strategy for playing on the edges $V_x \times V_y$ such that \ref{p1}--\ref{p3} hold for $H$ with
$$
    \gamma' = \frac{\gamma^3}{2^{v(H_e) + 6} T^2},\; T' = 2^{2v(H_e)+1} T,\; \text{ and } c' = \min\left\{ c, \gamma^2 / (2^{v(H_e) + 6} T) \right\}.
$$

Let $P \subseteq V_x \times V_y$ denote the set of all pairs of vertices which belong to a copy of $H_e$ in $G_e$, and
$$
    L = \left\{ \vv \in P \colon \deg_{G_e}(\vv, H_e) \ge N_{G_e}(H_e) / (2|P|) \right\}.
$$
We use the following two observations to derive an estimate on the size of $L$:
\begin{equation} \label{eq:sum_deg}
    \sum_{\vv \in P} \deg_{G_e}(\vv, H_e) = N_{G_e}(H_e) \stackrel{\ref{p1}}{\ge} \gamma n^{v(H_e)} p^{e(H_e)},
\end{equation}
and
\begin{equation} \label{eq:sum_deg2}
    \sum_{\vv \in P} \deg_{G_e}(\vv, H_e)^2 \le 2 \sum_{\{x,y\} \subseteq J \subseteq H_e} N_{G_e}(H_e^J) \stackrel{\ref{p3}}{\le} 2^{v(H_e)} T \frac{n^{2v(H_e)} p^{2e(H_e)}}{n^{v(F)}p^{e(F)}},
\end{equation}
where the second sum in \eqref{eq:sum_deg2} goes over induced subgraphs and $\{x, y\} \subseteq F \subseteq H_e$ is an induced subgraph which minimizes $n^{v(F)} p^{e(F)}$. Consider a random variable $X = \deg_{G_e}(\vv, H_e)$ for $\vv \in P$ chosen uniformly at random. Using the Paley-Zygmund inequality and $\mathbb{E}[X] = N_{G_e}(H_e) / |P|$ (follows from the first equality in \eqref{eq:sum_deg}), we get:
\begin{align}
    |L| = |P| \cdot \Pr[\vv \in L] &= |P| \cdot \Pr[\deg_{G_e}(\vv, H_e) \ge \mathbb{E}[X]/2] \nonumber \\
    &\ge (|P| / 4) \cdot \frac{(\mathbb{E}[X])^2}{\mathbb{E}[X^2]} \stackrel{\eqref{eq:sum_deg}, \eqref{eq:sum_deg2}}{\ge} \frac{\gamma^2}{2^{v(H_e) + 2} T} n^{v(F)} p^{e(F)}. \label{eq:L_size}
\end{align}

The number of copies of $F$ in $G_e$ is clearly an upper-bound on the size of $P$, as every pair of vertices $\mathbf{v} \in V_x \times V_y$ which belongs to $H_e$ necessarily belongs to $F$, thus by \ref{p2} we have $|P| \le Tn^{v(F)} p^{e(F)}$. Therefore, by the definition of $L$ and \eqref{eq:sum_deg}, for every $\vv \in L$ we have
\begin{equation} \label{eq:deg_L}
    \deg_{G_e}(\vv, H_e) \ge \frac{\gamma}{2T} n^{v(H_e) - v(F)} p^{e(H_e) - e(F)}.
\end{equation}
From \eqref{eq:L_size} and \eqref{eq:deg_L} we conclude that offering edges in $L$, in any order, results in  $\Omega(n^{v(H)} p^{e(H)})$ copies of $H$ regardless of which edges are chosen. Note that it is crucial here that $L$ is larger than $b$, which follows from the upper bound on $b$. It remains to show that Waiter can offer a fraction of these edges such that \ref{p2} and \ref{p3} hold as well.

Let $C = T2^{v(H_e) + 2}$. For each induced $H' \subseteq H_e$ which contains $\{x,y\}$, set
\begin{equation} \label{eq:p2}
    B_2(H') = \left\{ \vv \in V_x \times V_y \colon \deg_{G_e}(\vv, H') \ge C n^{v(H')} p^{e(H')} / |L| \right\}.
\end{equation}
By the inductive assumption \ref{p2}, we have $|B_2(H')| \le 2^{-(v(H_e)+2)} |L|$. Similarly, for each induced $J \subseteq H_e$ (with no restriction on the containment of $x$ and $y$), set
\begin{equation} \label{eq:p3}
    B_3(J) = \left\{ \vv \in V_x \times V_y \colon \deg_{G_e}(\vv, H_e^J) \ge C n^{2v(H_e) - v(J)} p^{2e(H_e) - e(J)} / |L| \right\}.
\end{equation}
By \ref{p3}, we have $|B_3(J)| \le 2^{-(v(H_e)+2)} |L|$. Therefore, the set
$$
    L' = L \setminus \left( \bigcup_{\{x,y\} \subseteq H' \subseteq H_e} B_2(H') \cup \bigcup_{J \subseteq H_e} B_3(J) \right)
$$
is of size $|L'| \ge |L|/2$. Note that as long as Waiter only plays on $L'$,  \ref{p2} is satisfied for every $H' \subseteq H$ and \ref{p3} is satisfied for those $J \subseteq H$ which contain both $x$ and $y$ (note that here we indeed mean $H$ and not $H_e$).

Let us now, finally, continue playing the game. In each of the next $|L'|/(2b)$ rounds, Waiter chooses $b$ new pairs from $L'$ such that, for every induced $J \subseteq H$ which does not contain both $x$ and $y$ (thus it is also induced in $H$), each offered edge $e$ satisfies
\begin{equation} \label{eq:G_plus_e}
\deg_{G + e}(e, H^J) \le 2^{v(H_e)+1} C \cdot n^{2v(H_e) - v(J)} p^{2e(H_e) - e(J) + 1} / |L'|,
\end{equation}
where $G$ denotes the current graph (i.e.\ $G_e$ and all the edges from $L'$ chosen by Client so far). If this is possible, that is, there are always $b$ such free edges, then after $|L'|/(2b) \ge 1$ rounds the resulting graph $G$ satisfies all three required properties.

It remains to show that such $b$ edges are always available. Suppose, towards a contradiction, that at some point there are more than $|L'| 2^{-(v(H_e)+1)}$ edges which violate \eqref{eq:G_plus_e} for some induced $J$ with $\{x, y\} \not \subseteq J$. Let us denote such \emph{bad} edges with $B$, and let us denote the set of edges in $L'$ chosen by Client so far with $Q$. We now count the number $h_J$ of copies of $H^J$ in $G_e \cup Q \cup B$ with one edge corresponding to $(x,y)$ in $Q$ and the other in $B$, in two ways: Once from the point of view of $B$ and once from $Q$. We get
$$
     |B| \cdot 2^{v(H_e)+1}C n^{2v(H_e) - v(J)} p^{2e(H_e) - e(J) + 1} / |L'| \stackrel{\eqref{eq:G_plus_e}}{<} h_J \stackrel{\eqref{eq:p3}}{\le} |Q| \cdot Cn^{2v(H_e) - v(J)} p^{2e(H_e) - e(J)} / |L|,
$$
and the contradiction follows from $|B| \ge |L'|2^{-(v(H_e)+1)}$,  $|Q| < |L'| / (2b)$, and $|L'| \ge |L|/2$. This finishes the proof.
\end{proof}

\section{Concluding remarks} \label{sec:concluding}
The fact that it was an upper bound on the number of certain subgraphs that forced the copies of $H_e$ to be sufficiently uniformly distributed in the proof of Theorem \ref{thm:main_ind} might seem surprising at first, but it is a common theme in the study of pseudo-random graphs. The most prominent result of this kind, credited to Chung, Graham, and Wilson \cite{chung89quasi}, states that if a graph $G$ with density $p$ contains at most $(p^4 + o(1))n^4$ copies of labelled $4$-cycles, then $G$ enjoys a very strong edge-discrepancy property. The idea used in our case, namely controlling the distribution of the copies of a graph through an upper bound on the number of copies of its possible overlaps, to the best of the author's knowledge goes back to the seminal result of R\"odl and Ruci\'nski \cite{rodl95ramseythreshold} and was instrumental in a breakthrough by Schacht \cite{schacht16extremal} on extremal properties of random discrete structures. It is interesting to notice that the proof of Theorem \ref{thm:main_ind} is the first application of these ideas in the regime where the resulting graph has a density far below the density needed for copies of $H^e$ to cover $\Theta(n^2)$ pairs of vertices in $V_x \times V_y$.

Without any modifications, Theorem \ref{thm:main_ind} also works in the case of $k$-uniform hypergraphs. Another interesting example are $k$-term arithmetic progressions. Here the game is played on the set $\{1, \ldots, n\}$, where in each turn Waiter chooses $b$ numbers which have not previously been chosen, and Client picks one and adds it to the initially empty set $S$. The goal of Waiter is to maximise the number of $k$-term arithmetic progressions in $S$, and the goal of Client is to minimize this quantity. If $S$ was a randomly chosen subset of $n/b$ integers, then the threshold for the appearance of a $k$-term arithmetic progression in $S$ is $b = n^{2/k}$. On the one hand, the proof of \cite[Theorem 1.3 (i)]{bednarsa14picker}, based on the potential method, shows that for $b \gg n^{2/k}$ Client has a strategy which avoids $k$-term progressions. On the other hand, a straightforward modification of calculations in Theorem \ref{thm:main_ind} shows that for $b \ll n^{2/k}$ Waiter has a strategy which ensures $\Omega(n^2 / b^k)$ such progressions.

In general, to specify a Waiter-Client game we need a (finite) set $X$ and a family $\cF$ of subsets of $X$. The game proceeds in rounds, where in each round Waiter chooses $b$ new elements from $X$ and  Client picks one and adds it to the set $S$. The goal of Waiter is to maximise the number of sets $F \in \cF$ which are contained in $S$, and the goal of Client is to minimise this quantity. In the case of $H$-games, the set $X$ corresponds to the edges of $K_n$ and $\cF$ corresponds to those subsets of the edges which give a graph isomorphic to $H$. In the case of arithmetic progressions we naturally have $X = \{1, \ldots, n\}$ and $\cF$ are subsets of numbers which form a $k$-term arithmetic progression. Even though the two arguably most interesting cases, small subgraphs and arithmetic progressions, follow from the approach presented in this paper, it would be interesting to obtain a generalisation of Theorem \ref{thm:main_ind} which would capture sufficient properties of $\cF$ and directly apply in the general setting. A statement along these lines was obtained by Beck \cite[Theorem 1]{beck02picker} in the case $b=2$.

\textbf{Acknowledgment.} The author would like to thank Danny Hefetz for many helpful comments on the manuscript.

{\small \bibliographystyle{abbrv} \bibliography{graph_games}}

\end{document}